\documentclass{article}

\usepackage{amsmath,amssymb,amsthm}
\usepackage{hyperref}

\newtheorem{theorem}{Theorem}
\newtheorem{lemma}{Lemma}
\newtheorem{corollary}{Corollary}

\usepackage{graphicx} 
\usepackage[utf8]{inputenc}
\usepackage{makecell}
\usepackage{algorithm,algorithmic} 
\usepackage{multirow}
\usepackage{amsmath,amssymb}
\usepackage{mathtools}
\usepackage{array}
\usepackage{booktabs}
\usepackage{enumitem}
\usepackage[caption=false]{subfig} 
\usepackage{upgreek}
\usepackage{bm}
\usepackage{relsize}
\usepackage[capitalize]{cleveref}
\usepackage{xcolor}
\usepackage[normalem]{ulem}
\usepackage{caption}

\usepackage{listings}
\definecolor{codegreen}{rgb}{0,0.6,0}
\definecolor{codegray}{rgb}{0.5,0.5,0.5}
\definecolor{codepurple}{rgb}{0.58,0,0.82}
\definecolor{backcolour}{rgb}{0.95,0.95,0.92}

\crefname{lstlisting}{listing}{listings}
\Crefname{lstlisting}{Listing}{Listings}

\lstdefinestyle{mystyle}{
    backgroundcolor=\color{backcolour},   
    commentstyle=\color{codegreen},
    keywordstyle=\color{magenta},
    numberstyle=\tiny\color{codegray},
    stringstyle=\color{codepurple},
    basicstyle=\ttfamily\footnotesize,
    breakatwhitespace=false,         
    breaklines=true,                 
    captionpos=b,                    
    keepspaces=true,                 
    numbers=left,                    
    numbersep=5pt,                  
    showspaces=false,                
    showstringspaces=false,
    showtabs=false,                  
    tabsize=2
}

\newcommand{\tr}{\operatorname{tr}}
\newcommand{\ts}{^T}
\newcommand{\Rtheta}{R_{\theta}}

\DeclareMathOperator{\Diag}{diag}

\DeclareMathOperator{\est}{est}

\DeclareMathOperator*{\argmin}{arg\,\min}

\newcommand{\KW}{Karlson-Wald\'{e}n~}

\newtheorem{assumption}{Condition}
\newtheorem{example}{Example}
\newtheorem{remark}{Remark}

\title{A Variational Equation and Lower Bound for the Linear Least-Squares Backward Error}
\author{Eric Hallman\thanks{Contact: \href{eric.r.hallman@gmail.com}{\texttt{eric.r.hallman@gmail.com}}. Palo Alto, USA. ORCID: 0000-0001-7908-2296}}
\date{\today}

\begin{document}

\maketitle
    \begin{changemargin}{1cm}{1cm}
    \begingroup
    \footnotesize
    {\bf Abstract}: This paper derives a new variational equation for the linear least-squares backward error by expressing the backward error in terms of a generalized eigenvalue problem and using results from indefinite linear algebra. For problems with multiple right-hand sides, the variational equation also shows that the backward error can be decomposed as a sum of smaller backward error problems. Applications to stopping criteria for iterative methods are considered, and a new sketching-based lower bound is proposed which is provably of quality comparable to the sketched Karlson-Wald\'{e}n estimate.
    \endgroup
    \end{changemargin}


\begingroup
\endgroup

\begingroup
\endgroup

\section{Introduction}
Given an approximate solution $X$ to the linear least-squares problem 
\begin{equation*}
    \min_{X\in \mathbb{R}^{n\times d}} \|AX-B\|_F,
\end{equation*}
where $A\in \mathbb{R}^{m\times n}$ and $B\in \mathbb{R}^{m\times d}$, this paper considers the problem of efficiently estimating the {\em backward error} 
\begin{equation}\label{def:backward_error}
    \mu \equiv \min_{E,F}\left\{\|[E, \theta F]\|_F\,:\, (A+E)\ts (A+E)X = (A+E)\ts (B+F)\right\}.
\end{equation}
Here, $E$ and $F$ are perturbations such that $X$ is an exact solution to the perturbed least-squares problem, and $\theta \in (0,\infty]$ is a weighting parameter. In principle, the backward error can be used in stopping criteria for iterative least-squares solvers, where the solver halts once $\mu$ falls below a user-defined tolerance. In practice, computing $\mu$ directly requires solving a large eigenvalue problem and so is prohibitively expensive in this context. It is therefore desirable to find estimates that are both cheap and reliable; see \cite{karlson1997estimation, jiranek2010estimating, gratton2013simple, hallman2020estimating,epperly2026fast} for some prior work on this topic.

A key term appearing in expressions for the backward error is 
\begin{equation*}
    R_\theta \equiv R\begin{bmatrix}
        X \\ -\theta^{-1}I
    \end{bmatrix}^\dagger,\quad \text{where}\quad R\equiv B-AX,
\end{equation*}
where $A^\dagger$ denotes the Moore-Penrose inverse of $A$. It represents the optimal backward error for the {\em consistent} problem, 
\begin{equation*}
    R_\theta = \argmin_{[E,\theta F]} \left\{\|[E, \theta F]\|_F\,:\, (A+E)X = B+F\right\},
\end{equation*}
provided a valid perturbation exists. This requires the following condition:
\begin{assumption}\label{assumption:rank}
    The approximate solution $X$ satisfies
    $\mathcal{R}(R\ts)\subseteq \mathcal{R}([X\ts , \theta^{-1} I])$.
\end{assumption}
It is sufficient to have $\theta < \infty$, or for $X$ to have full column rank. Under this condition, the backward error may be expressed as 
\begin{equation}\label{def:mu_min_proj}
    \mu = \mu(A,R_\theta) \equiv \min_{Y} \left(\|YY^\dagger A\|_F^2 + \|(I-YY^\dagger)R_\theta\|_F^2\right)^{\frac{1}{2}}. 
\end{equation}
If \cref{assumption:rank} is not satisfied, the analysis is slightly more complicated \cite[Thm.~5.1]{sun1996optimal} but still requires solving a problem of the same type as \cref{def:mu_min_proj}. We find it useful to study the binary function $\mu(A,R_\theta)$ as opposed to the definition \cref{def:backward_error} from which it was derived, and so will use the term ``backward error'' to refer to $\mu(A,R_\theta)$ for the remainder of the paper.

\begin{remark}\label{remark:mu_rotation_invariant}
    It can be seen from \cref{def:mu_min_proj} that $\mu(A,R_\theta)$ is invariant under right rotations of its inputs. We can therefore assume without loss of generality that $A$ and $R_\theta$ both have full column rank. Accordingly, sometimes $R(\theta^{-2}I+X\ts X)^{-1/2}$ is used in place of $R_\theta$ in the literature, or $\theta r/\sqrt{1 + \theta^2\|x\|_2^2}$ for problems with a single right-hand side. 
\end{remark}


The main contribution of this work is the following theorem, which elegantly decomposes the backward error as a sum of smaller backward error problems:

\begin{theorem}\label{thm:main_decomposition}
    Let $A \in \mathbb{R}^{m\times n}$ and $\Rtheta\in \mathbb{R}^{m\times d}$. Then with $\mu(A,\Rtheta)$ defined as in \cref{def:mu_min_proj} and $k\equiv \min(n, d)$,
    \begin{equation}\label{eqn:main_decomposition}
        \mu(A,R_\theta) = \max_{\substack{P\ts P = I_{k}\\ Q\ts Q = I_{k}}}\left(\sum_{i=1}^{k}\mu^2(Ap_i, R_\theta q_i)\right)^{\frac{1}{2}},
    \end{equation}
    where $P = [p_1,\ldots, p_{k}]\in \mathbb{R}^{n\times k}$ and $Q=[q_1,\ldots,q_k]\in \mathbb{R}^{d\times k}$. 
\end{theorem}
Each summand can be evaluated cheaply and stably; \cref{thm:rank_two_eigenvalue_problem} gives an exact expression.

We also consider some potential applications of \cref{thm:main_decomposition} to stopping criteria for iterative methods. The main point of comparison in this regard is the {\em sketched \KW estimate} \cite{epperly2026fast}, which uses subspace embedding techniques to estimate the highly accurate \KW estimate \cite{karlson1997estimation, gratton2012accuracy}. The room for improvement is admittedly limited, as the sketched \KW estimate was found in \cite{epperly2026fast} to be highly accurate in practice. Still, there is one notable difference: any estimate derived from \cref{eqn:main_decomposition} is a lower bound on the backward error, while the sketched \KW estimate is not necessarily an upper or lower bound.

The rest of the paper is organized as follows: \cref{sec:background} provides background including notations and some lemmas. \Cref{sec:indefinite_optimization} proves an intermediate result that involves expressing the backward error in terms of a generalized eigenvalue problem. \Cref{sec:main_results} gives the proof of the main result, along with some results for problems with a single right-hand side. \Cref{sec:application_error_bounds} considers applications to error bounds and stopping criteria. \Cref{sec:numerical_experiments} includes the results of some numerical experiments, and \Cref{sec:conclusion} offers our concluding remarks.

\section{Background}\label{sec:background}

It was shown in \cite[Thm.~4.2]{sun1996optimal} that the backward error can be expressed as
\begin{equation}\label{def:mu_binary}
    \mu(A, R_\theta) = \left(\|R_\theta\|_F^2 + \tr_{-}(AA\ts  - R_\theta R_\theta\ts)\right)^{\frac{1}{2}},
\end{equation}
where $\tr_{-}(X)$ denotes the sum of the negative eigenvalues of $X$. This value can be attained in \cref{def:mu_min_proj} by having $YY^\dagger$ project onto the negative eigenspace of $AA\ts - R_\theta R_\theta\ts$. This matrix has rank at most $n+d$. If $m \gg n \gg d$ then the cost of computing $\mu(A,R_\theta)$ can be reduced using a QR factorization or SVD of $A$ \cite{walden1995optimal,karlson1997estimation}, but this is generally not practical in the context of iterative solvers.

Three simple upper bounds on $\mu(A,R_\theta)$ are obtained by evaluating \cref{def:mu_min_proj} for $Y=0$, $Y=A$, and $Y=R_\theta$. The first choice yields $\mu \leq \|R_\theta\|_F$, but this is not an especially useful upper bound since the residual does not necessarily converge to zero. The second is too costly to compute to be used for stopping rules, although efforts to estimate it are made in \cite{jiranek2010estimating}. The third equals $\|A\ts r\|_2/\|r\|_2$ for problems with a single right-hand side, and it is used in stopping rules for iterative solvers LSQR \cite{paige1982lsqr} and LSMR \cite{fong2011lsmr}. However, it is possible for these three bounds to simultaneously overestimate $\mu$ by a factor depending on the condition number of $A$ \cite{gratton2013simple}. 

The proofs in this work are often simplified by assuming that $[A,R_\theta]$ has full column rank, then invoking a continuity argument. This argument can be made precise via the following lemma:
\begin{lemma}\label{lemma:continuity_argument}
    The backward error as defined in \cref{def:mu_min_proj} satisfies
    \[
    \lim_{\epsilon \rightarrow 0}\mu\left(\begin{bmatrix}
        A \\ \epsilon I \\ 0
    \end{bmatrix}, \begin{bmatrix}
        R_\theta \\ 0 \\ \epsilon I
    \end{bmatrix}\right) = \mu(A, R_\theta). 
    \]
\end{lemma}
\begin{proof}
    Use \cref{def:mu_binary} and the fact that eigenvalues are continuous.
\end{proof}

\subsection{The Karlson-Wald\'{e}n estimate}
The \KW estimate was derived in \cite[(2.6)]{karlson1997estimation} as the solution to a maximization problem
\begin{align*}
    \nu(A, r_\theta) &\equiv \max_{\|p\|_2=1} \frac{|p\ts A\ts r_\theta|}{(\|Ap\|_2^2 + \|r_\theta\|_2^2)^{1/2}} \\
    &= \|(A\ts A+\|r_\theta\|_2^2I)^{-1/2}A\ts r_\theta\|_2, \nonumber
\end{align*}
where the maximum is attained by $p\propto (A\ts A + \|r_\theta\|_2^2I)^{-1}A\ts r_\theta$. For multiple right-hand sides, the estimate can be defined additively: if $W=[w_1,\ldots,w_d]$ is a set of right singular vectors for $R_\theta$, then
\begin{equation}\label{def:kw_multiple_rhs}
    \nu(A, R_\theta) \equiv \left(\sum_{i=1}^d \nu^2(A, R_\theta w_i)\right)^{\frac{1}{2}}.
\end{equation}
It can alternately be defined in terms of a minimal perturbation \cite[(14)]{hallman2020estimating}
\begin{equation*}
    \nu(A, R_\theta) = \min_{E,F} \{\|[E,F]\|_F\ :\ A\ts F + E\ts R_\theta = -A\ts R_\theta\},
\end{equation*}
which shows that $\nu(A,R_\theta)$, like $\mu(A,R_\theta)$, is symmetric in its inputs and invariant under right rotations of the inputs.

The \KW estimate is always accurate, satisfying \cite[Thm.~4.8]{hallman2020estimating}
\begin{equation*}
    1 \leq \frac{\mu(A,R_\theta)}{\nu(A,R_\theta)} \leq \sqrt{1 + \|AA^\dagger R_\theta R_\theta ^\dagger\|_2} \leq \sqrt{2}. 
\end{equation*}
The $\sqrt{2}$ bound cannot be improved in general \cite{gratton2012accuracy}. For the single right-hand side case the asymptotic accuracy of $\nu(A,r_\theta)$ was proved in \cite[Thm.~4.8]{grcar2003optimal}, and \cite[Cor.~3.2]{gratton2012accuracy} gives an asymptotic bound that is tighter than the one stated above. 

\subsubsection{The sketched Karlson-Wald\'{e}n estimate}

The \KW estimate is still too expensive to use directly for stopping rules, though some efforts to efficiently compute or estimate it have been made in \cite{jiranek2010estimating, hallman2018lsmb, fong2011minimum}. One practical approach, recently proposed in \cite{epperly2026fast}, is to use a {\em sketching matrix}: a matrix $S\in \mathbb{R}^{n_{\text{sketch}}\times m}$ with the property that $\|SAy\|_2\approx \|Ay\|_2$ for all $y\in \mathbb{R}^n$ with high probability. One option for $S$ is a Gaussian embedding, whose entries are independent $\mathcal{N}(0,n_{\text{sketch}}^{-1})$ random variables. Other options are discussed in \cite{Martinsson_Tropp_2020}; the authors in \cite{epperly2026fast} recommend using a class known as sparse sign embeddings.

Given a sketching matrix $S\in \mathbb{R}^{n_{\text{sketch}}\times m}$ ($n_{\text{sketch}} \geq n$), the {\em sketched \KW estimate} is defined as 
\begin{equation}\label{def:sketched_kw_estimate}
    \widetilde{\nu}(A,r_\theta; S) \equiv \|((SA)\ts (SA) + \|r_\theta\|_2^2I)^{-1/2}A\ts r_\theta\|_2
\end{equation}
for the single right-hand-side case; the definition can be naturally extended to the multiple right-hand-side case as with \cref{def:kw_multiple_rhs}. This estimate can be practical for iterative methods because the SVD of $SA$ can be precomputed, after which computing $\widetilde{\nu}$ requires only computing $A\ts r_\theta$ plus $\mathcal{O}(n^2)$ operations. It was observed in \cite{epperly2026fast} that the sketched estimate is extremely reliable in practice, although the authors note that it is difficult to determine at runtime the precise level of distortion incurred by the sketch.



\subsection{Tools from indefinite linear algebra}
The key technique in this work is to transform the eigenvalue problem from \cref{def:mu_binary} into a generalized eigenvalue problem. We first introduce a standard result for the generalized eigenvalue problem, essentially \cite[Cor.~8.7.2]{GoVa13} with minor modifications.
\begin{lemma}\label{lemma:gen_evp_basic}
    Let $A\in \mathbb{R}^{n\times n}$ be symmetric positive definite and $B\in \mathbb{R}^{n\times n}$ symmetric and nonsingular. Then there exists a nonsingular $V = [v_1,\ldots,v_n]\in \mathbb{R}^{n\times n}$ such that
    \[V\ts AV = \Diag(a_1,\ldots, a_n)\quad \text{and}\quad V\ts BV = \Diag(b_1,\ldots,b_n).\]
    Moreover, $Av_i = \lambda_i Bv_i$, where $\lambda_i = a_i/b_i$. 
\end{lemma}
The columns of $V$ are known as {\em generalized eigenvectors of the pencil $(A,B)$}, and the $\lambda_i$ are the {\em generalized eigenvalues}.

Next, define
\begin{equation*}
    J_{n,d} \equiv \begin{bmatrix}
        I_n & 0\\ 0 & -I_d
    \end{bmatrix}.
\end{equation*}
The second lemma does the heavy lifting, and is a special case of a result by Kova\v{c}-Striko and Veseli\'{c} \cite[Thm.~3.1]{kovavc1995trace}.
\begin{lemma}\label{lemma:striko_trace}
    Let $A\in \mathbb{R}^{(n+d)\times (n+d)}$ be symmetric positive semidefinite. Then
    \begin{equation*}
        \min_{X\ts J_{n,d}X = -I_d} \tr (X\ts AX) \geq -\sum_{i=1}^d\lambda_i^{-},
    \end{equation*}
    where 
    \begin{equation*}
       \lambda_n^+\geq \cdots \geq \lambda_1^+ \geq  \lambda_1^- \geq \cdots \geq \lambda_d^-
    \end{equation*} are the generalized eigenvalues of $(A, J_{n,d})$. If there exists a feasible matrix $X_*$ whose columns are generalized eigenvectors corresponding to $\lambda_1^{-},\ldots,\lambda_{d}^-$, then equality is attained at $X_*$. 
\end{lemma}

Finally, any matrix $X$ satisfying $X\ts J_{n,d}X=-I_d$ has a {\em hyperbolic CS decomposition}, an indefinite analogue of the CS decomposition for unitary matrices \cite[Thm.~2.5.3]{GoVa13}. The lemma below is adapted from  \cite[Thm.~3.2]{higham2003jorthogonal}.
\begin{lemma}\label{lemma:hyperbolic_CS}
    Assume that $n\geq d$. If $X\in \mathbb{R}^{(n+d)\times d}$ satisfies $X\ts J_{n,d}X = -I_d$, then there exist orthogonal matrices $Q, Z\in \mathbb{R}^{d\times d}$ and a matrix $P\in \mathbb{R}^{n\times d}$ with orthonormal columns such that
    \begin{equation*}
     X = \begin{bmatrix}
         P & 0 \\ 0 & Q
     \end{bmatrix}\begin{bmatrix}
         S \\ C
     \end{bmatrix}Z\ts, 
    \end{equation*}
    where $S=\Diag(s_i)$ and $C=\Diag(c_i)$ satisfy ${C^2 - S^2=I_d}$. Conversely, any $X$ having this factorization satisfies $X\ts J_{n,d}X=-I_d$.
\end{lemma}

\section{Optimization over an indefinite manifold} \label{sec:indefinite_optimization}
This section presents and proves an intermediate result that admits a lower bound on the backward error. As with the expression in \cref{def:mu_binary} it can suffer from stability issues if evaluated directly, but it is useful in proving the main result in \cref{thm:main_decomposition}. 

\begin{lemma}\label{lemma:variational_first}
For any $A\in \mathbb{R}^{m\times n}$ and $R_\theta\in \mathbb{R}^{m\times d}$,
\begin{equation}\label{eqn:variational_first}
    \mu(A, R_\theta) = \sup_{X\ts J_{n,d}X = -I_d} \left(\|R_\theta\|_F^2 - \|[A, R_\theta]X\|_F^2\right)^{\frac{1}{2}}.
\end{equation}
\end{lemma}

\begin{proof}
    Assume to start that $[A,R_\theta]$ has full column rank, so that $[A,R_\theta]\ts[A,R_\theta]$ is symmetric positive definite, and consider the pencil $([A,R_\theta]\ts[A,R_\theta], J_{n,d})$. By \cref{lemma:gen_evp_basic}, there exists a nonsingular $V\in \mathbb{R}^{(n+d)\times (n+d)}$ such that $V\ts J_{n,d}V = J_{n,d}$\footnote{Without loss of generality, the columns of $V$ can be permuted and scaled and so that the diagonal entries of $V\ts J_{n,d}V$ are $\pm 1$. Preservation of inertia implies that there are exactly $n$ positive and $d$ negative eigenvalues.} and 
    \begin{equation}\label{eqn:genl_eig_V}
        [A,R_\theta]\ts [A,R_\theta]V = J_{n,d}V\Lambda, \quad \Lambda \equiv \Diag(\lambda_{n}^+,\ldots,\lambda_1^+,\lambda_1^-,\ldots,\lambda_{d}^-).
    \end{equation}
    The generalized eigenvalues are ordered so that
    \[\lambda_n^+\geq \cdots \geq \lambda_1^+ > 0 > \lambda_1^- \geq \cdots \geq \lambda_d^-.\]
    Applying \cref{lemma:striko_trace} to the positive definite matrix $[A,R_\theta]\ts[A,R_\theta]$ then implies that
    \begin{equation}\label{eqn:negative_genl_eigvals}
        \min_{X\ts J_{n,d}X = -I_d} \|[A,R_\theta]X\|_F^2 = -\sum_{i=1}^d\lambda_i^{-},
    \end{equation}
    where equality is attained by setting the columns of $X$ to be the generalized eigenvectors corresponding to $\lambda_{1}^-,\ldots,\lambda_d^-$. 
    
    Finally, the generalized eigenvalues of $([A,R_\theta]\ts[A,R_\theta], J_{n,d})$ are precisely the nonzero eigenvalues of $AA\ts-R_\theta R_\theta\ts$. To see this, left-multiply both sides of \cref{eqn:genl_eig_V} by $[A,R_\theta]J_{n,d}$ to get
    \begin{equation}\label{eqn:U_relation_AV}
        (AA\ts-R_\theta R_\theta \ts)U = U\Lambda,\quad U \equiv [A,R_\theta]V,
    \end{equation}
    where the columns of $U$ are linearly independent since $[A,R_\theta]$ has full column rank. Thus, the $n+d$ columns of $U$ account for all nonzero eigenvectors of the rank-($n+d$) matrix $AA\ts - R_\theta R_\theta\ts$.    
    
    Starting from \cref{def:mu_binary}, we thus find that
    \begin{align*}
        \mu^2(A,R_\theta) &= \|R_\theta\|_F^2 + \tr_-\left(AA\ts - R_\theta R_\theta\ts\right) \\
        &= \|R_\theta\|_F^2 + \sum_{i=1}^d\lambda_i^{-} \\
        &= \max_{X\ts J_{n,d}X = -I_d} \|R_\theta\|_F^2 - \|[A, R_\theta]X\|_F^2.
    \end{align*}
    The case where $[A,R_\theta]$ does not have full column rank is proved by making a continuity argument as in \cref{lemma:continuity_argument}.
\end{proof}

Equality in \cref{eqn:variational_first} is not necessarily attainable when $[A,R_\theta]$ does not have full column rank because the set of matrices $X$ satisfying $X\ts J_{n,d}X = -I_d$ is not bounded. 

\begin{example}
    Let $A = R_\theta = [1]$, 
    in which case $\mu(A,R_\theta) = 1$. Equality cannot be attained for any $X$, but we can come arbitrarily close by taking $X=[s,-\sqrt{1+s^2}]\ts$ as $s\rightarrow \infty$.
\end{example}

\begin{remark}
    The eigenvalue relations for the matrix $AA\ts - R_\theta R_\theta \ts$ and the pencil $([A,R_\theta]\ts[A,R_\theta], J_{n,d})$ can be expressed simultaneously by the indefinite generalized eigenvalue problem
    \begin{equation*}
        \begin{bmatrix}
            0 & [A,R_\theta]\ts \\
            [A,R_\theta] & 0
        \end{bmatrix}\begin{bmatrix}
            v \\ u
        \end{bmatrix} = \xi \begin{bmatrix}
            J_{n,d} & 0 \\ 0 & I_m
        \end{bmatrix}\begin{bmatrix}
            v \\ u
        \end{bmatrix}.
    \end{equation*}
    The generalized eigenvalues $\xi$ satisfy $\xi^2=\lambda$, and so are either real or purely imaginary.
\end{remark}

\subsection{Eigenvalue expressions}
For problems with a single right-hand side, it is known \cite[(2.15)]{walden1995optimal} that the backward error can be expressed in terms of a singular value problem: 
\begin{equation*}
    \mu(A, r_\theta) = \min\left\{\|r_\theta\|_2, \sigma_{\min}\left(\left[A, \|r_\theta\|_2(I-r_\theta r_\theta^\dagger)\right]\right)\right\}.
\end{equation*}
Although this expression avoids forming $A\ts A$ or $AA\ts$ explicitly, it overstates the size of the problem if considered naively: the matrix whose smallest singular value is to be computed is $m\times (n+m)$. But since the backward error is rotation-invariant, the QR factorization of $[A,r_\theta]$ can be used to reduce the size to $(n+1)\times (2n+1)$. 

The true size of the problem is better captured by the generalized eigenvalue problem \cref{eqn:genl_eig_V}. For problems with a single right-hand side, it may be formulated as follows:
\begin{theorem}\label{thm:gevp_expression_mu}
    For $A\in \mathbb{R}^{m\times n}$ and $r_\theta\in \mathbb{R}^m$, 
    \begin{equation*}
        \mu(A,r_\theta) = \lambda_{\min}\left(\begin{bmatrix}
            A\ts A + \|r_\theta\|_2^2I & A\ts r_\theta \\
            r_\theta\ts A & 0
        \end{bmatrix}, J_{n,1}\right)^\frac{1}{2}.
    \end{equation*}
\end{theorem}
\begin{proof}
    When $[A,r_\theta]$ has full column rank, take the eigenvalue relations in the proof of \cref{lemma:variational_first} and shift by $\|r_\theta\|_2^2J_{n,1}$. When $[A,r_\theta]$ does not have full column rank, use a continuity argument as in \cref{lemma:continuity_argument}.
\end{proof}
Since the generalized eigenvalues in \cref{thm:gevp_expression_mu} are solutions to the equation 
\[
    \det\left(\begin{bmatrix}
        A\ts A + (\|r_\theta\|_2^2 - \mu^2)I & A\ts r_\theta \\
        r_\theta\ts A & \mu^2
    \end{bmatrix}\right) = 0,
\]
we can recover an expression for the backward error that appears in \cite[(3.13)]{gratton2012accuracy}. 
\begin{corollary}\label{cor:backward_error_fixed_point}
    For $A\in \mathbb{R}^{m\times n}$ and $r_\theta\in \mathbb{R}^m$, $\mu(A,r_\theta)$ is the smallest nonnegative number solving
    \begin{equation*}
        \mu^2 = r_\theta\ts A(A\ts A + (\|r_\theta\|_2^2 - \mu^2)I)^{-1}A\ts r_\theta.
    \end{equation*}
\end{corollary}
Considered as a function of $\mu^2$, the right-hand side is convex and monotonically increasing. The \KW estimate is the square root of the value of this function at zero. Given the right singular vectors of $A$, the equation can be solved using fixed-point iteration or Newton's method, and so the backward error can be computed to a high degree of accuracy.

\section{Main results} \label{sec:main_results}
From looking at \cref{lemma:variational_first}, it might be tempting to split the expression in \cref{eqn:variational_first} columnwise, with each column of $X$ being one generalized eigenvector. As it turns out, it is more natural to consider the hyperbolic CS decomposition of $X$. The proof of \cref{thm:main_decomposition} does exactly that, and as a result decomposes the backward error into a sum of smaller backward error problems. 

\begin{proof}[Proof of \cref{thm:main_decomposition}]
     Since \cref{def:mu_min_proj} implies that $\mu(A,R_\theta) = \mu(R_\theta, A)$, we may assume without loss of generality that $n\geq d$. Proceeding from \cref{lemma:variational_first}, consider the hyperbolic CS decomposition of any feasible $X$. By \cref{lemma:hyperbolic_CS}, optimizing over $X$ is equivalent to optimizing over its constituent parts $P, Q, S, C, Z$. Thus, 
    \begin{align*}
        \mu^2(A,R_\theta) &= \sup_{X\ts J_{n,d}X=-I_d}\|R_\theta\|_F^2 - \|[A,R_\theta]X\|_F^2 \\
        &= \sup_{P,Q,S,C,Z} \|R_\theta\|_F^2 - \left\|[AP,R_\theta Q]\begin{bmatrix}S\\C\end{bmatrix}Z\ts\right\|_F^2 \\
        &= \sup_{P,Q,S,C} \|R_\theta Q\|_F^2 - \left\|[AP,R_\theta Q]\begin{bmatrix}S\\C\end{bmatrix}\right\|_F^2 \\
        &= \sup_{P,Q,S,C} \sum_{i=1}^{d}\left(\|R_\theta q_i\|_2^2 - \left\|[Ap_i,R_\theta q_i]\begin{bmatrix}s_i\\c_i\end{bmatrix}\right\|_2^2\right) \\
        &= \max_{P,Q} \sum_{i=1}^{d}\left(\sup_{c_i^2-s_i^2=1}\|R_\theta q_i\|_2^2 - \left\|[Ap_i,R_\theta q_i]\begin{bmatrix}s_i\\c_i\end{bmatrix}\right\|_2^2\right) \\
        &= \max_{P,Q} \sum_{i=1}^d \mu^2(Ap_i, R_\theta q_i),
    \end{align*}
    where the final step uses \cref{lemma:variational_first}. Note that the maximum is attainable because the set of feasible $\{P, Q\}$ (both matrices with orthonormal columns) is compact. 
\end{proof}

A secondary result presents a clean expression for each summand in \cref{eqn:main_decomposition}.

\begin{theorem}\label{thm:rank_two_eigenvalue_problem}
    For $a\in \mathbb{R}^m$ and $r \in \mathbb{R}^m$, at least one of which is nonzero, 
    \begin{equation*}
        \mu(a,r) = \frac{2 |a\ts r|}{\|a+r\|_2 + \|a-r\|_2},
    \end{equation*}
    where $\mu(a,r)$ is as defined in \cref{def:mu_min_proj}.
\end{theorem}
\begin{proof}
    By \Cref{cor:backward_error_fixed_point}, $\mu(a,r)$ is the smallest nonnegative solution to the equation
    \begin{equation*}
        \mu^2 = \frac{(a\ts r)^2}{\|a\|_2^2 + \|r\|_2^2 - \mu^2}.
    \end{equation*}
    Solving the resulting quadratic equation yields 
    \begin{align*}
        \mu(a,r) &= \left(\frac{\|a\|_2^2 + \|r\|_2^2 -\sqrt{(\|a\|_2^2 + \|r\|_2^2)^2 - 4(a\ts r)^2}}{2}\right)^{\frac{1}{2}} \\
        &= \left(\frac{\|a\|_2^2 + \|r\|_2^2 - \|a+r\|_2\|a-r\|_2}{2}\right)^{\frac{1}{2}} \\
        &= \left(\frac{\|a+r\|_2^2 + \|a-r\|_2^2 - 2\|a+r\|_2\|a-r\|_2}{4}\right)^{\frac{1}{2}}\\
        &= \frac{|\|a+r\|_2 - \|a-r\|_2|}{2}.
    \end{align*}
    Although elegant, this formula is unstable. The stable formulation is
    \begin{equation*}
        \mu = \frac{|\|a+r\|_2 - \|a-r\|_2|}{2}\cdot \frac{\|a+r\|_2 + \|a-r\|_2}{\|a+r\|_2+\|a-r\|_2} = \frac{2|a\ts r|}{\|a+r\|_2+\|a-r\|_2},
    \end{equation*}
    which completes the proof.
\end{proof}

For problems with a single right-hand side, \cref{thm:main_decomposition} means that the backward error can be ``explained'' by a single direction $Ap$. The next theorem gives an expression for the optimal $p$. 

\begin{theorem}\label{thm:main_single_rhs}
    Let $A\in \mathbb{R}^{m\times n}$ and $r_\theta \in \mathbb{R}^m$. If $\mu(A, r_\theta) > 0$, then
    \begin{equation}\label{eqn:main_single_rhs}
        \mu(A,r_\theta) = \max_{\|p\|_2=1}\frac{2|p\ts A\ts r_\theta|}{\|Ap+r_\theta\|_2 + \|Ap-r_\theta\|_2},
    \end{equation}
    and the maximum is attained by $p\propto \left(A\ts A + (\|r_\theta\|_2^2 - \mu^2)I\right)^{-1}A\ts r_\theta$.
\end{theorem}
\begin{proof}
    The expression \cref{eqn:main_single_rhs} is simply the result of combining \cref{thm:main_decomposition} with \cref{thm:rank_two_eigenvalue_problem}. As for the optimal value of $p$, it can be checked that
    \begin{equation*}
        v = \begin{bmatrix}
            (A\ts A + (\|r_\theta\|_2^2 - \mu^2)I)^{-1}A\ts r_\theta \\
            -1
        \end{bmatrix}
    \end{equation*}
    is a generalized eigenvector of the pencil $([A,r_\theta]\ts[A,r_\theta], J_{n,1})$ with eigenvalue $\lambda_1^{-} = \mu^2 - \|r_\theta\|_2^2$. By \cref{lemma:striko_trace} and its application to \cref{thm:main_decomposition}, the maximum of $\mu(Ap,r_\theta)$ is attained precisely by taking the CS decomposition of $v$ (i.e., normalizing the vector in the first entry to get $p$). 
\end{proof}

\section{Application to error bounds} \label{sec:application_error_bounds}
This section will focus on problems with a single right-hand side; the more general case will be left to future work.


\subsection{Lower bound} \label{sec:lower_bound}

Given some estimate $\mu_{\text{est}}$ for the backward error, one may use a sketch $SA$ to compute
\begin{equation}\label{def:p_sketched}
    \widetilde{p} \equiv \left((SA)\ts (SA) + (\|r_\theta\|_2^2 - \mu_{\text{est}}^2)I\right)^{-1}A\ts r_\theta,
\end{equation}
then normalize $\widetilde{p}$ and apply \cref{thm:main_single_rhs} to obtain a true lower bound. As compared with computing $\widetilde{\nu}$ as in \cref{def:sketched_kw_estimate} it requires one additional matvec with $A$ to compute $Ap$. The choice $\mu_{\est}=0$ will generally be sufficient, as the quality of the sketch is far more likely to be factor limiting the accuracy of any subsequent backward error estimate.

The following theorem gives bounds on the error of this estimator; the proof is given in \Cref{apx:lower_bound_accuracy} of the appendix.
\begin{theorem}\label{thm:lower_bound_accuracy}
    Let $\widetilde{p}$ be as in \cref{def:p_sketched}, with $\mu_{\text{est}} =0$ and where the sketch $S$ satisfies $(1-\eta)\|Ay\|_2\leq \|S(Ay)\|_2\leq (1+\eta)\|Ay\|_2$ for all $y\in \mathbb{R}^n$. Then
    \begin{equation*}
        \frac{1-\eta^2}{1+\eta^2}\nu(A,r_\theta) \leq \mu(A\widetilde{p}/\|\widetilde{p}\|_2, r_\theta) \leq \mu(A, r_\theta).
    \end{equation*}
\end{theorem}

\subsubsection{Iterative refinement}
The lower bound may be improved via iterative refinement. Given $\widetilde{p}$ and $\mu_{\text{est}}$ as in \cref{def:p_sketched}, one can compute
\begin{equation}\label{eqn:iterative_refinement}
    \widetilde{p}_{\text{next}} = \widetilde{p} + \left((SA)\ts (SA) + (\|r_\theta\|_2^2 - \mu_{\text{est}}^2)I\right)^{-1}(A\ts (r_\theta -A\widetilde{p}) - (\|r_\theta\|_2^2 - \mu_{\text{est}}^2)\widetilde{p}).
\end{equation}
This will improve the estimate if the sketch $SA$ is of sufficient quality. An iterative method could potentially use $\widetilde{p}$ as a starting guess to compute $\widetilde{p}_{\text{next}}$ on the subsequent iteration, but the added cost per iteration is non-negligible.

\subsubsection{Estimate recycling}\label{sec:recycled_bound}
Once $Ap$ is computed, its value may be reused for other choices of $r_\theta$, including for subsequent iterations of an iterative method. For fixed $Ap$, \cref{eqn:main_single_rhs} can be computed with just access to $r_\theta$, and it is not necessary to find $A\ts r_\theta$. The vector $p$ may be recomputed whenever the estimated error drops below a user-defined threshold. 

\subsection{Upper bound}
With $\widetilde{p}$ as in \cref{def:p_sketched}, one can interpret $[\widetilde{p}\ts, -1]\ts$ as an approximate generalized eigenvector (see \cref{thm:main_single_rhs}). From \cref{eqn:U_relation_AV} it follows that $\widetilde{u}\equiv A\widetilde{p}-r_\theta$ is an approximate eigenvector for $AA\ts - r_\theta r_\theta\ts$, and so in principle we can use \cref{def:mu_min_proj} to bound the backward error as
\begin{equation*}
    \mu(A,r_\theta) \leq \left(\|\widetilde{u}\widetilde{u}^\dagger A\|_2^2 + \|(I - \widetilde{u}\widetilde{u}^\dagger)r_\theta\|_2^2\right)^{\frac{1}{2}}.
\end{equation*}
More generously, we could compute an orthonormal basis $U$ for $[A\widetilde{p}, r_\theta]$ and directly solve the right-hand side of
\begin{equation}\label{eqn:upper_bound_generous}
    \mu(A,r_\theta) \leq \mu(U\ts A, U\ts r_\theta).
\end{equation}
These upper bounds require additional matvecs with $A$ beyond what is needed to compute the lower bound.


\section{Numerical Experiments} \label{sec:numerical_experiments}
This section presents the results of some numerical experiments. Experiments were done in Octave 6.2.0 on a 2020 MacBook Pro with Apple M1 chip. 

In each case, we solved the problem $\min_x \|Ax-b\|_2$ using LSMR \cite{fong2011lsmr}. The matrix $A$ was always the matrix \texttt{GL7d12} from the SuiteSparse Matrix Collection \cite{davis2011university, kolodziej2019suitesparse}, which is $8899\times 1019$ with $37519$ nonzero entries. The condition number of $A$ is approximately $3.5\times 10^{16}$, but its singular values lie in two clusters $(14.34,1.71)\cup (1.66\mathrm{e}-14,4.07\mathrm{e}-16)$. The right-hand side was chosen randomly, according to the formula
\begin{equation*}
    b \equiv Ax + 10^{-4}\|A\|_2 w,
\end{equation*}
where $x\sim \mathcal{N}(0,n^{-1}I_n)$ and $w\sim \mathcal{N}(0,m^{-1}I_m)$ had independent Gaussian entries. The LSMR algorithm was run to a tolerance of $\text{ATOL}=10^{-12}$, stopping when $\|A\ts r\|_2\leq \text{ATOL}\|A\|_F\|r\|_2$. We tested three different sizes for the Gaussian sketching matrix $S\in \mathbb{R}^{n_{\text{sketch}}\times m}$, using $n_{\text{sketch}}\in \{\lfloor1.5n\rfloor, 6n, 16n\}$\footnote{In the case $n_{\text{sketch}}=16n$, the ``sketched'' matrix $SA$ is in fact larger than the original matrix $A$. Nonetheless, the sketch does not exactly recover the singular values or right singular vectors of $A$.}.

Results are shown in \cref{fig:gl7_results}. In all cases, the following error estimates are presented:
\begin{itemize}
    \item A solid black line represents the true backward error.
    \item Dotted and dashed black lines respectively represent the basic upper bounds $\|r_\theta\|_2$ and $\|A\ts r_\theta\|_2/\|r_\theta\|_2$.
    \item Solid red lines represent the lower bound proposed in \cref{sec:lower_bound} and the upper bound \cref{eqn:upper_bound_generous}.
    \item A dashed red line represents the sketched \KW estimate \cref{def:sketched_kw_estimate}.
    \item A dot-dashed red line represents the lower bound after one step of iterative refinement, as in \cref{eqn:iterative_refinement}. The starting guess $\widetilde{p}$ was computed from scratch according to \cref{def:p_sketched} on each iteration. 
    \item A dotted red line represents the lower bound that comes from recycling $Ap$ as described in \cref{sec:recycled_bound}. The vector $p$ is recomputed whenever the bound falls below $10^{-12}$. 
\end{itemize}
The plots on the left show the backward error relative to $\|A\|_2$, and the plots on the right show the ratio of the backward error to the sketched \KW estimate (dashed), the lower bound (solid), and the lower bound after one step of iterative refinement (dot-dashed).

First, the good news:
\begin{itemize}
    \item The accuracy of the lower bound was always comparable to that of the sketched \KW estimate, and was more accurate in the case of the cheapest sketch.
    \item The recycled lower bound did not degrade too quickly compared to the true backward error. It may therefore be useful as a cheap error monitor in cases where $r_\theta$ is computed as part of the iterative method, with more accurate estimates being computed only when the error is close to the stopping tolerance.
\end{itemize}

Now, the less-good news:
\begin{itemize}
    \item Although iterative refinement of $\widetilde{p}$ did improve the quality of the lower bound, it only did so in the final case $n_{\text{sketch}}=16n$, which may be unnecessarily large (the experiments in \cite{epperly2026fast} use sparse sign embeddings with $n_{\text{sketch}}=12n$). Depending on the size and nature of the problem, the increased accuracy of the error estimate may not justify the added up-front computation.
    \item Attempts to perform iterative refinement using the $\widetilde{p}$ from the previous iteration as a starting guess were not particularly successful (experiments not shown).
    \item The upper bound was only marginally better than the readily available bound $\|A\ts r\|_2/\|r\|_2$, and still trailed the true backward error by several orders of magnitude.
\end{itemize}

\begin{figure}
    \centering
    \includegraphics[width=0.45\linewidth]{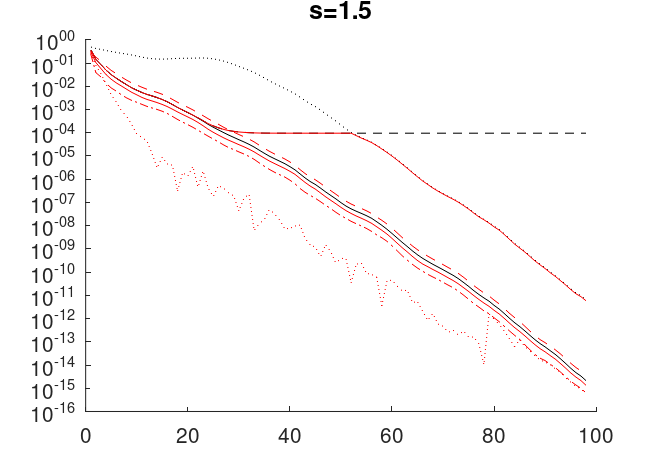}
    \includegraphics[width=0.45\linewidth]{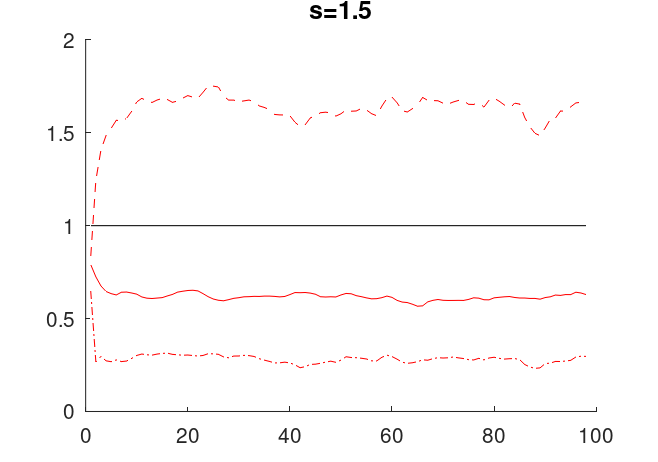}
        \includegraphics[width=0.45\linewidth]{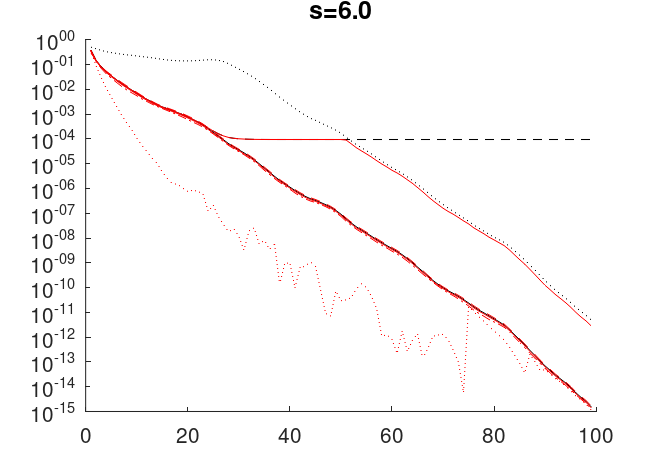}
    \includegraphics[width=0.45\linewidth]{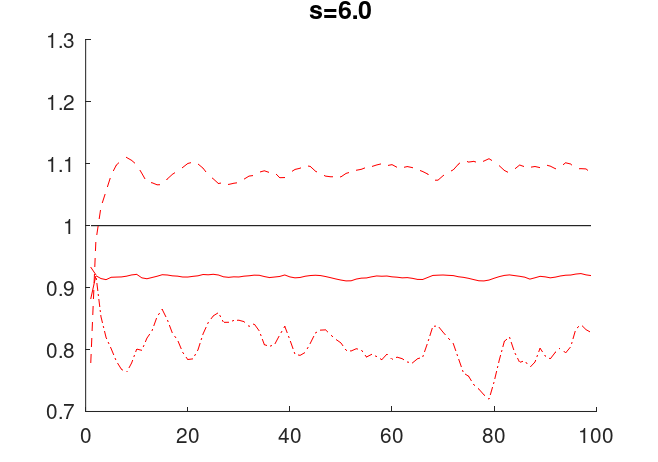}
        \includegraphics[width=0.45\linewidth]{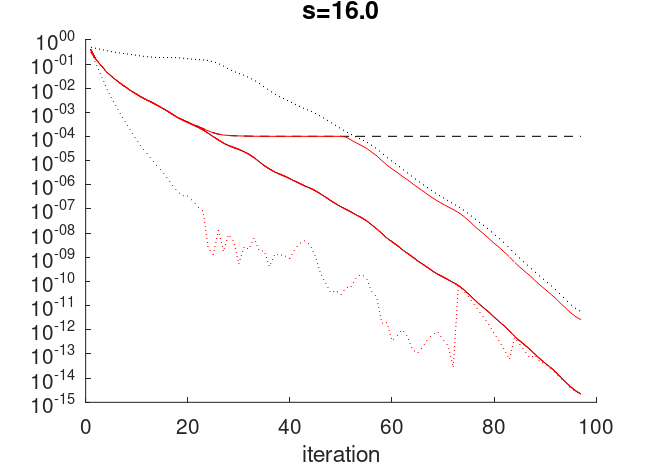}
    \includegraphics[width=0.45\linewidth]{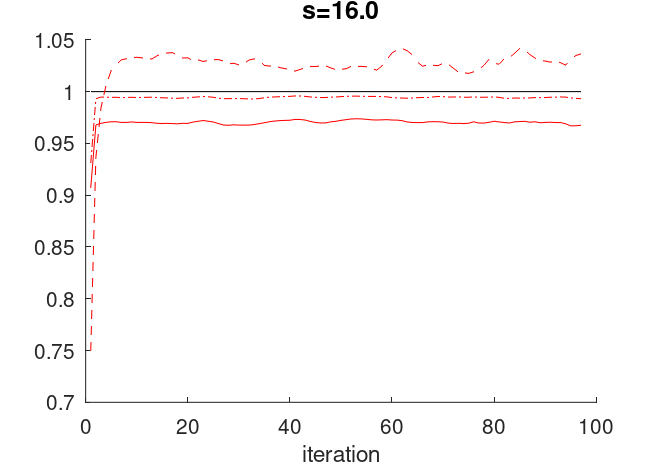}
    \caption{Relative backward error estimates (left) and ratios compared to the true backward error (right) with three different sketching parameters. The newly proposed lower bound (bold red line) is of comparable quality to the sketched \KW estimate (dashed red line).}
    \label{fig:gl7_results}
\end{figure}

\section{Conclusion} \label{sec:conclusion}

The primary contribution of this work is a theoretical one: a decomposition of the backward error into a sum of smaller terms, which uses the fact that the backward error can be expressed naturally in terms of a generalized eigenvalue problem. This decomposition allows for the simple computation of lower bounds on the backward error, and in particular can be used to convert a sketch-based estimate of the backward error into a lower bound of comparable quality. Efforts to produce upper bounds in a similar manner were generally unsuccessful. 

For a potential application to iterative methods, one can compute a single test vector $Ap$, which may be reused over multiple iterations to cheaply find lower bounds on the backward error. If the residual does not change too drastically from one iteration to the next then the bounds may be of sufficient quality to use in place of more expensive estimates.

One potential avenue for future work is to develop procedures for finding lower bounds when the problem has multiple right-hand sides. Whereas the single right-hand side case involves estimating a single vector $p$, the general case involves two matrices $P$ and $Q$, constrained to have orthonormal columns. It would, of course, be valuable to find a method of producing accurate upper bounds; failing that, future work could explain why finding a good upper bound appears to be more difficult than finding a good lower bound.

\section*{Statement on use of LLMs}
The manuscript and the code used in numerical experiments were both written entirely by the author. The idea to use the Kantorovich inequality in the proof of \cref{thm:lower_bound_accuracy} originated from a language model; the author is responsible for finding and checking the cited article \cite{henrici1961two}.

A language model also suggested that the decomposition underlying \cref{thm:main_decomposition} is known elsewhere in the literature, and repeatedly mentioned the terms ``Krein space'' and ``Pontryagin theory'', which concern linear algebra over an indefinite inner product. The author finds the language model's claim quite plausible, but has not yet been able to track down a specific source.

\appendix
\section{Proof of \cref{thm:lower_bound_accuracy}}\label{apx:lower_bound_accuracy}
The goal is to prove that
\begin{equation*}
    \frac{1-\eta^2}{1+\eta^2}\nu(A,r_\theta) \leq \mu(A\widetilde{p}/\|\widetilde{p}\|_2, r_\theta) \leq \mu(A, r_\theta).
\end{equation*}
The right-hand inequality follows directly from \cref{thm:main_decomposition}. As for the left-hand inequality, 

\begin{align*}
    \mu^2(A\widetilde{p}/\|\widetilde{p}\|_2, r_\theta) &\geq \nu^2(A\widetilde{p}/\|\widetilde{p}\|_2, r_\theta) \\
    &= \frac{(\widetilde{p}\ts A\ts r_\theta)^2/\|\widetilde{p}\|_2^2}{\|A\widetilde{p}\|_2^2/\|\widetilde{p}\|_2^2 + \|r_\theta\|_2^2} \\
    &= \frac{\widetilde{p}\ts A\ts r_\theta r_\theta\ts A\widetilde{p}}{\widetilde{p}\ts (A\ts A + \|r_\theta\|_2^2I)\widetilde{p}}.
\end{align*}
This Rayleigh quotient is maximized at $p\equiv (A\ts A + \|r_\theta\|_2^2I)^{-1}A\ts r_\theta$. To compare the value in the above expression to the maximum attainable value, make the substitutions
\begin{align*}
    A\ts r_\theta &= (A\ts A + \|r_\theta\|_2^2I)p, \\
    \widetilde{p} &= ((SA)\ts (SA) + \|r_\theta\|_2^2I)^{-1}(A\ts A + \|r_\theta\|_2^2I)p.
\end{align*}
To simplify the expressions, define $\hat{A}\equiv [A\ts, \|r_\theta\|_2I]\ts$ and take the Cholesky factorization $((SA)\ts (SA) + \|r_\theta\|_2I)^{-1} = \hat{R}\ts \hat{R}$. Proceeding, we find that
\begin{align*}
    \frac{\widetilde{p}\ts A\ts r_\theta r_\theta\ts A\widetilde{p}}{\widetilde{p}\ts (A\ts A + \|r_\theta\|_2^2I)\widetilde{p}}
    &= \frac{\left(p\ts(\hat{A}\ts\hat{A})(\hat{R}\ts \hat{R})^{-1}(\hat{A}\ts \hat{A})p\right)^2}{p\ts(\hat{A}\ts\hat{A})(\hat{R}\ts \hat{R})^{-1} (\hat{A}\ts \hat{A})(\hat{R}\ts\hat{R})^{-1}(\hat{A}\ts \hat{A})p} \\
    &= \frac{\left(p\ts\hat{A}\ts(\hat{A}\hat{R}^{-1}) (\hat{R}^{-T}\hat{A}\ts) \hat{A}p\right)^2}{p\ts\hat{A}\ts(\hat{A}\hat{R}^{-1}) (\hat{R}^{-T}\hat{A}\ts)(\hat{A}\hat{R}^{-1}) (\hat{R}^{-T}\hat{A}\ts) \hat{A}p} \\
    &= \frac{\left(p\ts\hat{A}\ts H \hat{A}p\right)^2}{p\ts\hat{A}\ts H^2 \hat{A}p},
\end{align*}
where $H \equiv (\hat{A}\hat{R}^{-1}) (\hat{R}^{-T}\hat{A}\ts)$. 

To bound this expression, we use the Kantorovich inequality \cite{henrici1961two}, which (when put in terms of matrix algebra) states that for a symmetric positive definite matrix $M$ and vector $x$,
\begin{equation*}
    (x\ts M x)(x\ts M^{-1}x) \leq \frac{(\lambda_{\max}(M)+\lambda_{\min}(M))^2}{4\lambda_{\max}(M)\lambda_{\min}(M)}(x\ts x)^2 = \frac{(1+\kappa(M))^2}{4\kappa(M)}(x\ts x)^2,
\end{equation*}
where $\kappa(M) = \lambda_{\max}(M)/\lambda_{\min}(M)$. Applying this inequality with $x = H^{1/2}\hat{A}p$ and $M = H$, then rearranging, yields
\begin{align*}
    \frac{\left(p\ts\hat{A}\ts H \hat{A}p\right)^2}{p\ts\hat{A}\ts H^2 \hat{A}p}
    &\geq \frac{4\kappa(H)}{(1+\kappa(H))^2}\|\hat{A}p\|_2^2 \\
    &= \frac{4\kappa(H)}{(1+\kappa(H))^2}r_\theta\ts A (A\ts A+\|r_\theta\|_2^2I)^{-1}A\ts r_\theta \\
    &= \frac{4\kappa(H)}{(1+\kappa(H))^2} \nu^2(A, r_\theta),
\end{align*}
where $\kappa(H) = \sigma_{\max}^2(\hat{A}\hat{R}^{-1})/\sigma_{\min}^2(\hat{A}\hat{R}^{-1})$.\footnote{Note that although $H$ is not necessarily positive definite, $\hat{A}p$ is in the column space of $H$. Thus only the nonzero singular values of $H$ are considered.}

Finally, from the sketching bounds $(1-\eta)\|Ay\|_2 \leq \|S(Ay)\|_2 \leq (1+\eta)\|Ay\|_2$ we can find that
\begin{equation*}
    (1-\eta)\|\hat{A}y\|_2 \leq \|\hat{R}y\|_2 \leq (1+\eta)\|\hat{A}y\|_2
\end{equation*}
holds as well; i.e., the regularizing term $\|r_\theta\|_2^2I$ can only reduce the distortion of the sketch. From \cite[Fact 2.2]{epperly2026fast} it follows that the condition number of $\hat{A}\hat{R}^{-1}$ is bounded above by $(1+\eta)/(1-\eta)$. Thus, 
\begin{align*}
    \mu(A\widetilde{p}/\|\widetilde{p}\|_2, r_\theta) &\geq
    \frac{2\sqrt{\kappa(H)}}{1+\kappa(H)} \nu(A, r_\theta)\\
    &\geq \frac{2\frac{1+\eta}{1-\eta}}{1 + \frac{(1+\eta)^2}{(1-\eta)^2}}\nu(A, r_\theta) \\
    &= \frac{2(1+\eta)(1-\eta)}{(1-\eta)^2 + (1+\eta)^2}\nu(A, r_\theta)\\
    &= \frac{1 - \eta^2}{1+\eta^2}\nu(A, r_\theta).
\end{align*}

\bibliography{references}
\bibliographystyle{abbrv}

\end{document}